\documentclass[12pt,fleqn]{amsart}

\usepackage{amssymb}

\newtheorem{theorem}{Theorem}[section]
\newtheorem{lemma}[theorem]{Lemma}

\newtheorem{problem}[theorem]{Problem}

\theoremstyle{definition}

\theoremstyle{remark}

\numberwithin{equation}{section}

\begin{document}

\title[Normalmeasure]{A normal measure on a compact connected space}

%    author two information
\author[G.\ Plebanek]{Grzegorz Plebanek}
\address{Instytut Matematyczny, Uniwersytet Wroc\l awski}
\email{grzes@math.uni.wroc.pl}
\thanks{Partially supported by NCN grant 2013/11/B/ST1/03596 (2014-2017).}

\subjclass[2010]{28C15, 54D05}

%    \subjclass is required.
%\subjclass[2000]{Primary 46B26, 46E15; Secondary 46E27.}

%\date{}

%"Communicated by" -- provide editor's name; required.
%\commby{Nigel Kalton}
%\subjclass[2010]{Primary 46B26, 46B03, 46E15.}

%\maketitle

%%%%%%%%%%%%%%%%%%%%%%%%%%%%%%%%%%%%%%%%%LOCAL SHORTENINGS
%%%%%%%%%%%%%%NUMBERS
\newcommand{\con}{\mathfrak c}
\newcommand{\eps}{\varepsilon}
%%%%%%%%%%%%%%%%%%%%FRAK FAMILIES
\newcommand{\alg}{\mathfrak A}
\newcommand{\algb}{\mathfrak B}
\newcommand{\algc}{\mathfrak C}
\newcommand{\ma}{\mathfrak M}
\newcommand{\pa}{\mathfrak P}
%%%%%%%%%%%%%%%%%%%%SCRIPT FAMILIES
\newcommand{\BB}{\protect{\mathcal B}}
\newcommand{\AAA}{\mathcal A}
\newcommand{\CC}{{\mathcal C}}
\newcommand{\FF}{{\mathcal F}}
\newcommand{\GG}{{\mathcal G}}
\newcommand{\LL}{{\mathcal L}}
\newcommand{\UU}{{\mathcal U}}
\newcommand{\VV}{{\mathcal V}}
\newcommand{\HH}{{\mathcal H}}
\newcommand{\DD}{{\mathcal D}}
\newcommand{\ZZ}{{\mathcal Z}}
\newcommand{\RR}{\protect{\mathcal R}}
\newcommand{\ide}{\mathcal N}
%%%%%%%%%%%%%%%%%%%%%%%SYMBOLS
\newcommand{\btu}{\bigtriangleup}
\newcommand{\hra}{\hookrightarrow}
\newcommand{\ve}{\vee}
\newcommand{\we}{\cdot}
\newcommand{\de}{\protect{\rm{\; d}}}
\newcommand{\er}{\mathbb R}
\newcommand{\qu}{\mathbb Q}
\newcommand{\supp}{{\rm supp} }
\newcommand{\card}{{\rm card} }
\newcommand{\wn}{{\rm int} }
\newcommand{\wh}{\widehat }
\newcommand{\ult}{{\rm ULT}}
\newcommand{\vf}{\varphi}
\newcommand{\osc}{{\rm osc}}
\newcommand{\ol}{\overline}
\newcommand{\me}{\protect{\bf v}}
\newcommand{\ex}{\protect{\bf x}}
\newcommand{\stevo}{Todor\v{c}evi\'c}
\newcommand{\cc}{\protect{\mathfrak C}}
\newcommand{\scc}{\protect{\mathfrak C^*}}
\newcommand{\lra}{\longrightarrow}
\newcommand{\sm}{\setminus}
\newcommand{\uhr}{\upharpoonright}
\newcommand{\en}{\mathbb N}
\newcommand{\sub}{\subseteq}
\newcommand{\ms}{$(M^*)$}
\newcommand{\m}{$(M)$}
\newcommand{\MA}{MA$(\omega_1)$}
\newcommand{\clop}{\protect{\rm Clop} }
%%%%%%%%%%%%%%%%%%%%%%%%%%%%%%%%%%%%%%%%%%%%%%%%%%%%%%%%%%%

\begin{abstract}
We present a construction of a compact connected space which supports a normal probability measure.
\end{abstract}

\maketitle

\section{Introduction}

If $K$ is a compact Hausdorff space then we denote by $P(K)$ the set of all probability regular Borel measures on $K$.
We write $\ZZ(K)$ for the family of all closed $G_\delta$ subsets of $K$. Since every compact space is normal, $Z\in\ZZ(K)$ if and only if $Z$ is a zero set,
i.e.\ $Z=f^{-1}(0)$ for some continuous function $f:K\to\er$.

A measure $\mu\in P(K)$ is
{\em normal} if $\mu$  is order-continuous on the Banach lattice $C(K)$.  Equivalently, $\mu(F)=0$ whenever $F\sub K$ is a closed set with empty interior
(\cite{DDLS}, Theorem 4.6.3).
A typical example of a normal measure is the natural measure defined on the Stone space of the measure algebra $\alg$ of the Lebesgue measure $\lambda$
on $[0,1]$. Since the algebra $\alg$ is complete, its Stone space is extremely disconnected.

By a result from \cite{FP64} if $K$ is a locally connected compactum then no measure $\mu\in P(K)$ can be normal, cf.\ \cite{DDLS}, Proposition 4.6.20.
The following problem was posed in \cite{FP64}, cf.\ \cite{FL80}.

\begin{problem}\label{problem}
Suppose that $K$ is a compact and $\mu\in P(K)$ is a normal measure. Must $K$ be disconnected?
\end{problem}

We show below that the answer is negative, namely we prove the following result.

\begin{theorem}\label{main}
There is a compact connected space $L$ of weight $\con$ which is the support of a normal measure.
\end{theorem}

 I wish to thank H.\ Garth Dales for the discussion concerning the subject of this note and for making the preliminary version of the forthcoming book \cite{DDLS} available to me.

\section{Preliminaries}

Recall that $\mu\in P(K)$ is said to be {\em strictly positive} or
{\em fully supported by}  $K$  if
$\mu(U)>0$ for every non-empty open set $U\sub K$.

\begin{lemma}\label{1:1}
Let $K$ be a compact space, and suppose that $\mu$ is a strictly positive measure on $K$ such that
$\mu(Z)=0$ for every $Z\in\ZZ(K)$ with empty interior. Then $\mu$ is a normal measure.
\end{lemma}

\begin{proof}
Assume that there is a closed set $F\sub K$ with empty interior but with $\mu(F)>0$. Then we derive a contradiction by the following observation.
\medskip

{\sc Claim.} Every closed set $F\sub K$ with empty interior is contained in some $Z\in\ZZ(K)$ with empty interior.
\medskip

Indeed, consider a maximal family $\FF$ of continuous functions $K\to [0,1]$ such that $f|F=0$ for $f\in\FF$ and $f\cdot g=0$ whenever $f,g\in\FF$, $f\neq g$.
Then $\FF$ is necessarily countable because $K$, being the support of a measure, satisfies the countable chain condition. Write $\FF=\{f_n: n\in\en\}$ and let
$f=\sum_n 2^{-n}f_n$ and  $Z=f^{-1}(0)$. Then the function $f$ is continuous so that $Z\in\ZZ(K)$. We have $Z\supseteq F$ and the interior of $Z$ must be empty by the maximality of $\FF$.
\end{proof}

If $f:K\to L$ is a continuous map and $\mu\in P(K)$ then the measure $f[\mu]\in P(L)$ is defined by
$f[\mu](B)=\mu(f^{-1}(B))$ for every Borel set $B\sub L$.

We shall consider inverse systems of compact spaces with measures of the form
\[\langle K_\alpha,\mu_\alpha, \pi^\alpha_\beta: \beta<\alpha<\kappa\rangle,\]
where  $\kappa$ is an ordinal number and for all $\gamma <\beta<\alpha<\kappa$ we have
\begin{itemize}\label{1:2}
\item[\bf 2(i)] $K_\alpha$ is a compact space and $\mu_\alpha\in P(K_\alpha)$;
\item[\bf 2(ii)] $\pi^\alpha_\beta:K_\alpha\to K_\beta$ is a continuous surjection;
\item[\bf 2(iii)]  $\pi^\beta_\gamma\circ \pi^\alpha_\beta=\pi^\alpha_\gamma$;
\item[\bf 2(iv)] $\pi^\alpha_\beta[\mu_\alpha]=\mu_\beta$.
\end{itemize}

The following summarises basic facts on inverse systems satisfying 2(i)-(iv).

\begin{theorem} \label{2:is}
Let $K$ be the limit of the system with uniquely defined continuous surjections $\pi_\alpha:K\to K_\alpha$ for
$\alpha<\kappa$.

\begin{itemize}
\item[(a)] $K$ is a compact space and $K$ is connected whenever  all the space $K_\alpha$ are connected.
\item[(b)] There is the unique $\mu\in P(K)$ such that $\pi_\alpha[\mu]=\mu_\alpha$ for $\alpha<\kappa$.
\item[(c)] If every $\mu_\alpha$ is strictly positive then $\mu$ is strictly positive.
\end{itemize}
\end{theorem}

Engelking's {\em General Topology} contains the topological part of \ref{2:is} (measure-theoretic ingredients call for a proper
reference). We also use the following fact on closed $G_\delta$ sets and inverse systems of length $\omega_1$.

\begin{lemma}\label{2:is2}
Let $K$ be the limit of an inverse system $\langle K_\alpha, \pi^\alpha_\beta: \beta<\alpha<\omega_1\rangle$. Then
for every $Z\in\ZZ(K)$, there are $\alpha<\omega_1$ and $Z_\alpha\in\ZZ(K_\alpha)$ with
$Z=\pi_\alpha^{-1}(Z_\alpha)$.
\end{lemma}

\begin{proof}
Sets of the form $\pi_\alpha^{-1}(V)$, where $\alpha<\kappa$ and $V\subseteq K_\alpha$ is open, give the canonical basis of $K$  (closed under countable unions).
Therefore if $Z\in\ZZ(K)$ then $Z=\bigcap_n \pi_{\alpha_n}^{-1}(V_n)$ for some $\alpha_n<\omega_1$ and some  open  $V_n\sub K_{\alpha_n}$.
Taking $\alpha>\sup_n \alpha_n$ we can write $Z=\bigcap_n\pi_\alpha^{-1}(W_n)$ for some open $W_n\sub K_\alpha$.
Let $Z_\alpha=\bigcap_n W_n$. Then $Z_\alpha$ is $G_\delta$ in $K_\alpha$, $\pi_\alpha^{-1}(Z_\alpha)=Z$ and $Z_\alpha=\pi_\alpha(Z)$ is closed.
\end{proof}

\section{Proof of Theorem \ref{main}}

We first describe a basic construction which will be used repeatedly.

\begin{lemma}\label{basic}
Let $K$ be a compact connected space, and let $\mu\in P(K)$ be a strictly positive measure. If $F\sub K$ is a closed set with $\mu(F)>0$, then there
are a compact connected space $\wh{K}$, a strictly positive measure $\wh{\mu}\in P(\wh{K})$ and a continuous surjection $f:\wh{K}\to K$ such that
$f[\wh{\mu}]=\mu$ and $\wn((f^{-1}(F))\neq\emptyset$.
\end{lemma}

\begin{proof}
Let $F_0$ be the support of $\mu$ restricted to $F$, that is
\[ F_0=F\sm\bigcup\{U: U\mbox{ open and } \mu(F\cap U)=0\}.\]
Let $\wh{K}=\{(x,t)\in K\times [0,1]: x\in F_0\mbox{ or } t=0\}$. Then $\wh{K}$ is clearly a compact connected space and $f(x,t)=x$ defines a continuous surjection
$f:\wh{K}\to K$. Moreover, the set $f^{-1}(F)$ contains $F_0\times [0,1]$, a set with non-empty interior. Hence $\wn (f^{-1}(F))\neq\emptyset.$

We can define $\wh{\mu}\in P(\wh{K})$ with the required property by setting
\[\wh{\mu}(B)=\mu(f(B\cap (K\sm F)\times\{0\}))+ \mu\otimes\lambda (F\times [0,1]\cap B),\]
for Borel sets $B\sub \wh{K}$, where $\lambda$ is the Lebesgue measure on $[0,1]$.
\end{proof}

\begin{lemma}\label{basic2}
Let $K$ be a compact connected space, and let $\mu\in P(K)$ be a strictly positive measure.
Then there are a compact connected space $K^\#$, a strictly positive measure ${\mu}^\#\in P({K^\#})$ and a continuous surjection $g: K^\#\to K$ such that
$g[\mu^\#]=\mu$ and $\wn((g^{-1}(Z))\neq\emptyset$ for every $Z\in\ZZ(K)$ with $\mu(Z)>0$.
\end{lemma}

\begin{proof}
Let $\{Z_\alpha:\alpha<\kappa\}$ be an enumeration of all sets $Z\in\ZZ(K)$ of positive measure. Setting $K_0=K$, $\mu_0=\mu$ we define
inductively  an inverse system $\langle K_\alpha,\mu_\alpha, \pi^\alpha_\beta: \beta<\alpha<\kappa\rangle$ satisfying 2(i)-(iv).
Assume the construction  for all $\alpha<\xi$.

If  $\xi$ is the limit ordinal we use Theorem \ref{2:is} and let $K_\xi$ be the limit of $K_\alpha$, $\alpha<\kappa$, and $\mu_\xi$ be the unique measure
as in \ref{2:is2}.

If $\xi=\alpha+1$ then we define $K_\xi$ and $\mu_\xi\in P(K_\xi)$ applying Lemma \ref{basic} to
$K=K_\alpha$, $\mu=\mu_\alpha$, $F=(\pi^\alpha_0)^{-1}(Z_\alpha)$.

Then we can define  $K^\#$ and $\mu^\#$ as the limit of $\langle K_\alpha,\mu_\alpha, \pi^\alpha_\beta: \beta\le\alpha<\kappa\rangle$
and set $g=\pi_0: K^\#\to K$.

Indeed, if $Z\in\ZZ(K)$ and $\mu(Z)>0$ then $Z=Z_\alpha$ for some $\alpha<\kappa$ so the interior of the set
\[(\pi_0^{\alpha+1})^{-1}(Z_\alpha)=(\pi^{\alpha+1}_\alpha)^{-1}((\pi_0^\alpha)^{-1}(Z_\alpha),\]
is nonempty by the basic construction of Lemma \ref{basic}. It follows that $\wn(g^{-1}(Z_\alpha))\neq\emptyset$, and we are done.
\end{proof}

We are now ready for the proof of Theorem \ref{main}. Let $L_0=[0,1]$ and $\mu_0=\lambda$.
Using Lemma \ref{basic2} we define an inverse system
$\langle L_\alpha,\mu_\alpha, \pi^\alpha_\beta: \beta\le\alpha<\omega_1\rangle$, where
$L_{\alpha+1}= (L_\alpha)^\#$ and $\mu_{\alpha+1}=(\mu_\alpha)^\#$.
Consider the limit $L$ of this inverse system with the limit measure $\nu\in P(L)$.

We shall check that $\nu$ is a normal measure using Lemma \ref{1:1}.
Take $Z\in\ZZ(L)$ with $\nu(Z)>0$.
It follows from Lemma \ref{2:is2} that $Z=\pi_\alpha^{-1}(Z_\alpha)$ for some $\alpha<\omega_1$ and $Z_\alpha\in\ZZ(L_\alpha)$.
Then the set
$(\pi^{\alpha+1}_\alpha)^{-1}(Z_\alpha)$
has non-empty interior in $L_{\alpha+1}=(L_\alpha)^\#$ and, consequently, $\wn(Z)\neq\emptyset$.

Note that in a compact space $K$  of topological weight $w(K)\le\con$ there are at most $\con$ many closed $G_\delta$ sets.  It follows from the proof of Lemma
\ref{basic2} that $w(K^\#)\le\con$ whenever $w(K)\le\con$. Therefore $w(L_\alpha)\le\con$ for every $\alpha<\omega_1$ and $w(L)=\con$.
This finishes the proof of our main result.

%\end{proof}
\bigskip

Let us remark that using Lemma \ref{basic} and the construction from Kunen \cite{Ku81} one can prove the following variant of Theorem \ref{main}.

\begin{theorem}\label{kunen}
Assuming the continuum hypothesis, there is a perfectly normal compact connected space $L$ supporting a normal probability measure.
\end{theorem}

Perfect normality of $L$  means that every closed subset of $L$ is $G_\delta$ so in particular the space $L$ from Theorem \ref{kunen} is first-countable.

\end{document}